\documentclass[12pt]{article}

\usepackage{amsmath,amssymb,amsthm,graphicx,url}

\newcommand\cA{{\mathcal A}}

\newcommand\cC{{\mathcal C}}

\newcommand\cL{{\mathcal L}}

\newcommand\cP{{\mathcal P}}

 \newcommand\cU{{\mathcal U}}

\newcommand\uu{{\mathbf u}}

\newcommand\vv{{\mathbf v}}

\newcommand\lul{{\mathbf l}}

\theoremstyle{plain}

\newtheorem{theorem}{Theorem}[section]

\newtheorem{lemma}[theorem]{Lemma}

\newtheorem{corollary}[theorem]{Corollary}

\newtheorem{claim}[theorem]{Claim}

\theoremstyle{definition}

\newcommand\lref[1]{Lemma~\ref{lem:#1}}

\newcommand\tref[1]{Theorem~\ref{thm:#1}}

\newcommand\cref[1]{Corollary~\ref{cor:#1}}

\newcommand\sref[1]{Section~\ref{sec:#1}}

\newcommand{\PG}{\mathrm{PG}}

\newcommand{\GF}{\mathrm{GF}}

{}

\begin{document}

\title{Search Problems in Vector Spaces}

\author{Tam\'as H\'eger\thanks{MTA--ELTE Geometric and Algebraic Combinatorics Research Group, H--1117 Budapest, P\'azm\'any P.\ s\'et\'any 1/C, Hungary, \textsc{hetamas@cs.elte.hu}. Research supported by OTKA Grant No.\ K~81310 and partially supported by ERC Grant No.\ 227701 DISCRETECONT.} 
\and Bal\'azs Patk\'os\thanks{Hungarian Academy of Sciences, Alfr\'ed R\'enyi Institute of Mathematics, P.O.B.\ 127, Budapest H--1364, Hungary, \textsc{patkos@renyi.hu}. Research supported by OTKA Grant PD-83586 and the J\'anos Bolyai Research Scholarship of the Hungarian Academy of Sciences}\and Marcella Tak\'ats\thanks{Department of Computer Science, E\"otv\"os Lor\'and University, H--1117 Budapest, P\'azm\'any P.\ s\'et\'any 1/C, Hungary, \textsc{takats@cs.elte.hu}. Research supported by OTKA Grant No.\ K~81310.}}

\maketitle

\begin{abstract}
We consider the following $q$-analog of the basic combinatorial search problem: let $q$ be a prime power and $\GF(q)$ the finite field of $q$ elements. Let $V$ denote an $n$-dimensional vector space over $\GF(q)$ and let $\mathbf{v}$ be an unknown 1-dimensional subspace of $V$. We will be interested in determining the minimum number of queries that is needed to find $\mathbf{v}$ provided all queries are subspaces of $V$ and the answer to a query $U$ is YES if $\mathbf{v} \leqslant U$ and NO if $\mathbf{v} \not\leqslant U$. This number will be denoted by $A(n,q)$ in the adaptive case (when for each queries answers are obtained immediately and later queries might depend on previous answers) and $M(n,q)$ in the non-adaptive case (when all queries must be made in advance).

In the case $n=3$ we prove $2q-1=A(3,q)<M(3,q)$ if $q$ is large enough. While for general values of $n$ and $q$ we establish the bounds
\[
n\log q \le A(n,q) \le (1+o(1))nq 
\]
and 
\[
(1-o(1))nq \le M(n,q) \le 2nq,
\]
provided $q$ tends to infinity. 
\end{abstract}



\section{Introduction}
The starting point of combinatorial search theory is the following problem: given a set $X$ of $n$ elements out of which one $x$ is marked, what is the minimum number $s$ of queries of the form of subsets $A_1,A_2,\ldots ,A_s$ of $X$ such that after getting to know whether $x$ belongs to $A_i$ for all $1 \le i \le s$ we are able to determine $x$. Since decades, the number $s$ is known to be equal to $\lceil \log n \rceil$ no matter if  the $i$th query might depend on the answers to the previous ones (\textit{adaptive search}) or we have to ask our queries at once (\textit{non-adaptive search}). (Here and throughout the paper $\log$ denotes the logarithm of base 2.)

There are lots of variants of this problem. There can be multiple marked elements and our aim can be to determine at least one of them or all of them or a constant fraction of them. The number of marked elements can be known or unknown. There can be restrictions on the possible set $Q$ of queries; only small subsets can be asked or other restrictions may apply. Also, there are models in between the adaptive and the non-adaptive version: we might be allowed to ask our queries in $r$ \textit{rounds}, that is our queries of the $i+1$st round may depend on the answers to all queries in the first $i$ rounds and we would like to minimize the total number of queries. For these and further models we refer the reader to the monograph of Du and Hwang \cite{DH}.

In this paper we address the $q$-analogue of the basic problem. Let $q$ be a prime power and $\GF(q)$ the finite field of $q$ elements. Let $V$ denote an $n$-dimensional vector space over $\GF(q)$ and let $\vv$ be a marked 1-dimensional subspace of $V$ (throughout the paper 1-dimensional subspaces will be denoted by boldface lower case letters, vectors will be denoted by lower case letters with normal typesetting and upper case letters will denote subspaces of higher or unknown dimension). We will be interested in determining the minimum number of queries that is needed to find $\vv$ provided all queries are subspaces of $V$ and the answer to a query $U$ is YES if $\vv \leqslant U$ and NO if $\vv \not\leqslant U$. This number will be denoted by $A(n,q)$ in the adaptive case and $M(n,q)$ in the non-adaptive case. Note that a set $\cU$ of subspaces of $V$ can be used as query set to determine the marked 1-space in a non-adaptive search if and only if for every pair $\uu, \vv$ of 1-subspaces of $V$ there exists a subspace $U \in \cU$ with $\uu \leqslant U, \vv \not\leqslant U$ or $\uu \not\leqslant U, \vv \leqslant U$. Such systems of subspaces are called \textit{separating}.

Note that the $q$-analogue problem fits into the original subset settings. Indeed, let the set of $k$-dimensional subspaces of an $n$-dimensional vector space $V$ over $\GF(q)$ be denoted by ${V \brack k}$. Its cardinality $|{V \brack k}|$ is 
\[
{n \brack k}_q={n \brack k}=\frac{(q^n-1)(q^{n-1}-1)\ldots(q^{n-k+1}-1)}{(q^k-1)(q^{k-1}-1)\ldots(q-1)}.
\] 
Then if we let the underlying set $X$ be ${V \brack 1}$ and the set $Q$ of allowed queries be 
\[
\left\{F \subset {V \brack 1}: \exists U \le V \textnormal{ with}\ F=\left\{\uu \in {V \brack 1}: \uu \le U\right\}\right\},
\]
then we obtain the same problem.

Let us note that it is easy to show that $A(n,2)=M(n,2)= n$ for all $n \ge 2$. The reader is welcome to think about the one line proof that we will describe in \sref{general}. Thus we will mainly focus on the case when $q \ge 3$.

The subspaces of an $n$-dimensional vector space over $\GF(q)$ are the elements
of the Desarguesian projective geometry $\PG(n-1,q)$. In \sref{proj} we
consider the case when $n$ equals 3, 
that is the case of projective planes. After introducing some projective
geometry terminology, we determine $A(3,q)$ for all 
prime powers $q$. 

\begin{theorem}
\label{thm:adproj}
Consider a projective plane $\pi_q$ of order $q$. Let $A(\pi_q)$ denote minimum number
of queries in adaptive search that is needed to determine a point of $\pi_q$
provided the queries can be either points or lines of $\pi_q$. With this
notation we have $A(\pi_q)\leq 2q-1$; if $q$ is a prime power, then $A(\PG(2,q))=2q-1$, that is the equality $A(3,q)=2q-1$ holds.
\end{theorem}

In \sref{proj}, we also address the problem of determining $M(3,q)$. We obtain
upper and lower bounds but not the exact value except if $q\geq121$ is a square. 
The most important consequence of our results is the following theorem that
states that the situation is completely 
different from that in the subset case where adaptive and non-adaptive search require the same number of queries.

\begin{theorem}
\label{thm:adnonad}
For $q\ge 9$ the inequality $A(3,q) <M(3,q)$ holds.
\end{theorem}

In \sref{general}, we address the general problem of giving upper and lower bounds on $A(n,q)$ and $M(n,q)$. Our main results are the following theorems.

\begin{theorem}
\label{thm:adgen}
For any prime power $q\ge 2$ and positive integer $n$ the inequalities $\log {n \brack 1}_q \le A(n,q) \le (q-1)(n-1)+1$ hold.
\end{theorem}

\begin{theorem}
\label{thm:nonadgen}
There exists an absolute constant $C>0$ such that for any positive integer $n$ and prime power $q$ the inequalities $\frac{1}{C}qn \le M(n,q) \le 2qn$ hold. Moreover, if $q$ tends to infinity, then $(1-o(1))qn\le M(n,q)$ holds.
\end{theorem}

We finish the Introduction by recalling the standard method to prove upper and lower bounds for adaptive search. In both cases we assume the existence of an Adversary. When showing a lower bound $b$ for the number of queries needed to determine the marked elements, we have to come up with a strategy that ensures that no matter what sequence of $b-1$ queries the Adversary asks we are able to answer these queries such that there exist at least two elements that match the answers given. In this way we make sure that $b-1$ queries are insufficient. When proving an upper bound our and the Adversary's roles change and this time our task is to provide a strategy using at most $b$ queries (depending on the answers of the Adversary) such that there exists exactly one element that matches the answers no matter what these answers are.

\section{Projective planes, the case $n=3$}
\label{sec:proj}

In this section we prove \tref{adproj} and \tref{adnonad}. Before describing the proofs let us introduce some terminology. For an overview on projective geometries over finite fields we refer to \cite{Hirsch}. Let $\pi$ be a projective plane of order $q$ with point set $\cP$ and line set $\cL$. We say that a point set $B$ is a \textit{blocking set} in $\pi$ if $|B \cap \ell| \geq 1$ for any line $\ell \in \cL$. A point $P$ of a blocking set $B$ is said to be \textit{essential} if $B \setminus \{P\}$ is not a blocking set. A set $\cC$ of lines \textit{covers} $\pi$ if $\cup_{\ell \in \cC}\ell=\cP$. A line $\ell$ of a cover $\cC$ is \textit{essential} if $\cC \setminus\{\ell\}$ is not a cover. A line $\ell$ is said to be a \textit{tangent} to a set $S \subseteq \cP$ if $|\ell \cap S|=1$ holds. Our main tool in proving \tref{adproj} is the following result.

\begin{theorem}[Blokhuis, Brouwer \cite{BlBr}]
\label{thm:szonyi} Let $S$ be a blocking set in $\PG(2,q)$. Then there are at least $2q + 1 - |S|$ distinct tangents to $S$ through any essential point of $S$.
\end{theorem}

This result is actually the same as a unique reducibility theorem of Sz\H{o}nyi \cite{Sz}; for more details, we refer to \cite{Harrach,BHSz}.
To obtain \cref{szline} from \tref{szonyi} observe that its statement is the dual of \tref{szonyi}.

\begin{corollary}
\label{cor:szline} Let $L$ be a set of covering lines in $\PG(2,q)$ and let $\ell$ be an essential line of $L$. Then the inequality  $|\ell\setminus \cup_{\ell' \neq \ell, \ell' \in L}\ell'| \ge 2q+1-|L|$ holds.
\end{corollary}

Now we recall and prove Theorem \ref{thm:adproj}.
\setcounter{section}{1}
\setcounter{theorem}{0}

\begin{theorem}
Consider a projective plane $\pi_q$ of order $q$. Let $A(\pi_q)$ denote minimum number
of queries in adaptive search that is needed to determine a point of $\pi_q$
provided the queries can be either points or lines of $\pi_q$. With this
notation we have $A(\pi_q)\leq 2q-1$; if $q$ is a prime power, then $A(\PG(2,q))=2q-1$, that is the equality $A(3,q)=2q-1$ holds.
\end{theorem}\setcounter{section}{2}\setcounter{theorem}{5}

\begin{proof}
To obtain the upper bound, let us consider the following simple algorithm. Let
$x$ be an arbitrary point of the plane and let $\ell_1,\ell_2, \ldots,
\ell_{q+1}$ be the lines containing $x$. Let us ask $\ell_1,\ldots ,\ell_q$ one
after the other. Once the Adversary answers YES, then we have to find the
unknown point on that particular line, this takes at most $q$ further
queries. Moreover, if the YES answer comes to a query $\ell_i$ with $i>1$,
then we only need at most $q-1$ queries as we already know that $x$ is not the
unknown point. Thus if a YES answer comes to the $i$th query, $1\leq i\leq q$, 
we are done using $1+q$ or $i+q-1\le 2q-1$ queries according to whether
$i=1$ or $i>1$. If all answers are NO, then we obtain that the unknown point 
is in $\ell_{q+1}\setminus \{x\}$ and thus we need at most $q+q-1=2q-1$ queries.

To obtain the lower bound let us assume first that the Adversary only asks
lines as queries. Note that our only choice is about when to say YES for the
first time. Indeed, if the queries are $\ell_1,\ell_2,\ldots ,\ell_k$ and $\ell_k$
is the first query which we answer with YES, then the best we can do from then
on is to say NO as many times as possible. As the only possibilities for the
unknown point are the points in $\ell_k \setminus \cup_{i=1}^{k-1}\ell_i$, 
therefore the maximum number of queries we can reach is $k+|\ell_k \setminus \cup_{i=1}^{k-1}\ell_i|-1$.

Our strategy is simple: let the $k$th query $\ell_k$ be the first one we
answer with YES if there exists a line $\ell$ such that
$\ell_1,\ell_2,\ldots ,\ell_k, \ell$ form a covering set of lines. Observe that if
an Adversary is able to identify the unknown point, then he must have received
a YES answer from us. Indeed, if not, then by our strategy, there would be more than
two points that are not contained in any of the lines and thus
undistinguishable. Let the $k$th be the first query to which we answered
YES. Then there exists a line $\ell$ such that $\ell_1,\ell_2,\ldots ,\ell_k,
\ell$ cover the projective plane. We claim that $\ell_k$ is essential. Indeed,
if not then we should have answered YES earlier. Therefore, \cref{szline}
applies with $\ell_k$ being an essential line of the covering set
$\{\ell_1,\ell_2,\ldots ,\ell_k, \ell\}$. Thus, by our observation in the previous
paragraph, the minimum number of queries needed is 
\[
k+|\ell_k \setminus \cup_{i=1}^{k-1}\ell_i|-1 \ge k+|\ell_k \setminus (\cup_{i=1}^{k-1}\ell_i \cup \ell)|-1 \ge k+2q+1-(k+1)-1=2q-1.
\]

Let us now consider the general case where the Adversary is allowed to ask
queries that are points. We will always try to replace a point query by a
line. If the $k$th query is a point $P_k$ and there is a line $\ell_k$
containing $P_k$ such that all previously queried lines and the lines that
replaced queried points cannot be extended by a single line to a cover, then
we answer NO and provide the additional information to the Adversary that the
unknown point does not lie in $\ell_k$. Following this strategy when the
Adversary asks a line and with one additional line we can obtain a cover, 
the reasoning of the previous paragraph goes through. 

It remains to check the case when a point $P_k$ is asked and for all lines
$P_k \in \ell \notin L=\{\ell_1,\ell_2,\ldots,\ell_{k-1}\}$ there exists another
line $\ell'$ such that $L \cup \{\ell,\ell'\}$ is a cover. 
In this case we may assume that $P_k \notin \ell'$ for the following reasons. Suppose not
and $P_k$ is the unique intersection point of $\ell$ and $\ell'$. Then as
neither $L\cup \{\ell\}$ nor $L \cup \{\ell'\}$ is a cover, there must exist
points $Q_{1} \in \ell\setminus (\cup_{\ell'' \in L}\ell''\cup \{P_k\}), Q_{2} \in \ell'
\setminus (\cup_{\ell'' \in L}\ell''\cup \{P_k\})$. Now for any line $\ell'''\neq \ell, \ell'$ that contains
$P_k$ the line $\ell^*$ that extends $L \cup \{\ell'''\}$ to a cover must
contain $Q_1$ and $Q_2$ and thus $\ell^*=\langle Q_1,Q_2\rangle$ and clearly
$P_k \notin \langle Q_1,Q_2\rangle$.

It also follows that $\ell^*$ is essential in
the cover $L \cup \{\ell''',\ell^*\}$, thus if we answer NO to the query $P_k$
and provide the additional information that the 
unknown point lies in $\ell^*$, then the calculation for the restricted case gives us the desired lower bound. 
\hfill\mbox{}\qed
\end{proof}

Let us now turn to the non-adaptive case. The following lemma states that it is enough to consider separating systems consisting of only lines.

\begin{lemma}
\label{lem:line}
For any separating system $S$ of a projective plane $\pi$, there exists another one $S'$ that contains only lines and $|S|=|S'|$ holds.
\end{lemma}

\begin{proof}
It is enough to prove the statement for minimal separating systems. 
Let $S$ be such a system such that it contains the minimum number of
points. If this number is 0, then we are done. Suppose $S$ contains a point
$P$. By minimality of $S$, we know that $S\setminus \{P\}$ is not
separating. Clearly, $P$ only separates pairs of points one of which is $P$
itself. There exists exactly one point $Q\neq P$ such that
$S\setminus \{P\}$ does not separate the pair $(P,Q)$. Indeed, by the above
there is at least 
one such point, furthermore if there was one more point $Q'$, then $Q$ and $Q'$
would not even be separated by $S$. Let $\ell$ be any line containing $P$ and
not containing $Q$. Then $S'=S\setminus \{P\} \cup \ell$ is 
a separating system such that $S'$ contains one point less than $S$. This contradicts the choice of $S$.
\hfill\mbox{}\qed
\end{proof}

Let us take a short graph theoretic detour. In a graph $G$ a subset $H_1
\subset V(G)$ of vertices \textit{resolves} another subset $H_2$ if the list
of path-distances in $G$ from the vertices in $H_1$ are unique in $H_2$,
i.e.\ for any $h_2,h_2' \in H_2$ there exists an $h_1 \in H_1$ such that
$d_G(h_1,h_2) \neq d_G(h_1,h_2')$ holds. A set $R$ of vertices is a
\textit{resolving set} in $G$ if it resolves $V(G)$.

For more information about resolving sets and related topics see \cite{BC}.

If $G$ is bipartite with
classes $A$ and $B$, then a subset $A'$ of $A$ ($B'$ of $B$) is
\textit{semi-resolving} if it resolves $B$ ($A$). Let $G_{\pi}$ be the
incidence graph of a projective plane $\pi$. Then by \lref{line} the minimum
size of a separating system in $\pi$ equals the minimum size of a
semi-resolving set in $G_{\pi}$. H\'eger and Tak\'ats \cite{HT} showed that
the minimum size of a resolving set in any projective plane of order $q\geq23$ is
$4q-4$ and obtained the following lower bound on the size of any
semi-resolving set in the incidence graph of $\PG(2,q)$. Note that $\tau_2(\pi)$
denotes the minimum size of a point set in $\pi$ that 
meets every line of $\pi$ in at least 2 points, that is $\tau_2(\pi)$ denotes the minimum size of a \textit{double (2-fold) blocking set} in $\pi$. 

\begin{theorem}[H\'eger, Tak\'ats \cite{HT}]
\label{thm:ht}
Let $S$ be a semi-resolving set in $\PG(2,q)$, $q \ge 3$. Then $|S| \ge \min\{2q +
q/4 - 3, \tau_2(\PG(2,q)) - 2\}$. 
\end{theorem}

\tref{ht} together with the following theorem implies \tref{adnonad}.

\begin{theorem}[Ball, Blokhuis \cite{BB}]
\label{thm:bb} Let $q \ge 9$. Then $\tau_2(\PG(2,q)) \ge 2(q +\sqrt q + 1)$, and equality holds if and
only if $q$ is a square.
\end{theorem}

On the other hand, Bailey \cite{Bailey} gave a semi-resolving set of size $\tau_2(\PG(2,q))-1$, and H\'eger and Tak\'ats \cite{HT} constructed one of size $2(q+\sqrt{q})$ in $\PG(2,q)$, $q$ a square prime power. 

\begin{corollary}
\label{cor:projprojproj}
Let $q \ge 121$ be a square. Then $M(3,q)=2q+2\sqrt q$ holds.
\end{corollary}

Recall that $A(3,q)=2q-1$ by Theorem \ref{thm:adproj}. Thus Corollary \ref{cor:projprojproj} and Theorem \ref{thm:adproj} together prove Theorem \ref{thm:adnonad}, which we recall below. \setcounter{section}{1} \setcounter{theorem}{1}
\begin{theorem}
For $q\ge 9$ the inequality $A(3,q) <M(3,q)$ holds.
\end{theorem}
\setcounter{section}{2} \setcounter{theorem}{7}

The exact value of $\tau_2(\PG(2,q))$ is not known in general. If $q>3$ is a prime, then Ball proved $\tau_2(\PG(2,q))\geq 2.5(q+1)$ \cite{Ball}. As for large square values of $q$ we have $M(3,q)/q=2+2/\sqrt{q}$, while for prime values of $q$ we have $M(3,q)/q> 2.5$, we obtain the following. 

\begin{theorem}\label{thm:nolimit}
The sequence $M(3,q)/q$ does not have a limit.
\end{theorem}

In case of $q=p^{2d+1}$, $p$ prime, $d\geq1$, Blokhuis, Storme and Sz\H{o}nyi \cite{BSSz} obtained the lower bound $\tau_2(\PG(2,q))\geq 2(q+1) + c_pq^{2/3}$, where $c_2=c_3=2^{-1/3}$ and $c_p=1$ otherwise. As for an upper bound if $q$ is not a square, Bacs\'o, H\'eger and Sz\H{o}nyi \cite{BHSz} showed $\tau_2(\PG(2,q))\leq 2q +2(q-1)/(r-1)$, where $q=r^d$, $r$ an odd prime power, $d$ odd. Thus for such parameters, \tref{ht} implies that if $r\geq11$, then $M(3,q)$ is either $\tau_2(\PG(2,q))-1$ or $\tau_2(\PG(2,q))-2$.

\section{General bounds}
\label{sec:general}
In this section we recall and prove \tref{adgen} and \tref{nonadgen}.

\setcounter{section}{1}\setcounter{theorem}{2}
\begin{theorem}
For any prime power $q\ge 2$ and positive integer $n$ the inequalities $\log {n \brack 1}_q \le A(n,q) \le (q-1)(n-1)+1$ hold.
\end{theorem}\setcounter{section}{3}\setcounter{theorem}{0}

\begin{proof} 
Let us begin with the lower bound as it follows from the trivial lower bound
that any separating system of subsets of $X$ should contain at least $\lceil \log |X| \rceil$ sets. Therefore any separating
system of subspaces of $V$ should contain at least $\lceil \log |{V \brack 1}|\rceil\ge (n-1)\log q$ subspaces. Note that if $q=2$, then the formula gives $\lceil \log 2^{n}-1\rceil=n$ as lower bound.

We will describe two algorithms to show the upper bound $A(n,q)\le (q-1)(n-1)+1$.
The first algorithm is a very simple inductive one and generalizes the algorithm that we had in the
projective plane case. First of all, note that if $n=2$, then the bound to prove is $q$ and just by asking $q$ out 
of the $q+1$ possible 1-subspaces we can determine the unknown 1-subspace $\uu$. Let us assume that for all $k<n$ 
we obtained an algorithm in the $k$ dimensional space that uses only $(k-1)$-subspaces as queries. 

Consider any $(n-2)$-subspace $U$ of $V$. There are exactly
$q+1$ $(n-1)$-subspaces $U_1,\ldots ,U_{q+1}$ of $V$ that contain $U$. Let us ask $q$ of them. After getting the answers 
to these queries, we will know whether $\uu \leqslant U$ or $\uu \subset U_i \setminus U$ holds for some $1 \le i \le q+1$ and in the latter case we even know the value of $i$. If $\uu \leqslant U$, then by induction we can finish our algorithm
in $(n-3)(q-1)+1$ queries that gives a total of $(n-2)(q-1)+2$ queries. If $\uu \subset U_i \setminus U$, then by our assumption that
an algorithm for the $(n-1)$ dimensional case uses only $(n-2)$-spaces, we can assume that the first query is $U$ and thus we need only
$(n-2)(q-1)$ more queries giving a total of $(n-1)(q-1)+1$ queries. Note that we can also satisfy the assumption that we only use $(n-1)$-subspaces, since, instead of querying an $(n-2)$-subspace $A$ of $U_i$, we can ask an $(n-1)$-subspace $A' \leqslant V$ such that $A'\cap U_i=A$. 

Note that even this easy algorithm does not utilize the whole power of adaptiveness as when decreasing the dimension by one, we can ask the $q$ queries at once. Thus the above algorithm uses at most $n-1$ rounds. In what follows, we introduce a two-round algorithm that uses the same number of queries to determine the unknown 1-subspace $\uu$.

Before describing the two-round algorithm note that to determine a 1-subspace $\uu$ it is enough to identify one
non-zero vector $u \in \uu$ as then $\uu=\{\lambda u: \lambda \in \GF(q)\}$. In
the next reasoning we will think of a vector $v \in V$ as an $n$-tuple of
elements of $\GF(q)$. For $i=1,2,\ldots ,n$ let us define the following
$(n-1)$-subspaces of $V$: $A_i=\{v=(v_1,v_2,\ldots ,v_n) \in V: v_i=0\}$. Let
$e_1,e_2,\ldots ,e_n$ denote the standard basis of $V$ and for $1 \le i<j\le n$
let us write $E_{i,j}=\langle e_i,e_j\rangle$. All $E_{i,j}$'s have dimension
2, therefore each of them contains $q+1$ 1-subspaces. Two of those are
$\{v=(v_1,v_2,\ldots ,v_n) \in E_{i,j}: v_i=0\}$ and $\{v=(v_1,v_2,\ldots ,v_n) \in
E_{i,j}: v_j=0\}$. For every pair $i,j$ let $\lul_{i,j,1},
\lul_{i,j,2},\ldots ,\lul_{i,j,q-1}$ be an arbitrary enumeration of the $q-1$ other
1-subspaces of $E_{i,j}$. Finally, for any $1\le i<j\le n$ and $1 \le k
\le q-1$ let us write $L_{i,j,k}=\{v=(v_1,v_2,\ldots ,v_n) \in V:
(0,\ldots ,0,v_i,0,\ldots ,0,v_j,0,\ldots ,0) \in \lul_{i,j,k}\}$. 
Clearly, all $L_{i,j,k}$'s are $(n-1)$-subspaces of $V$.

In the first round, our algorithm asks all subspaces $A_i$, $i=1,2,\ldots ,n$ as queries. Let $Z$ and $NZ$ denote
  the set of coordinates for which the answer was YES and NO,
  respectively. (Note that if $q=2$, then we are done as with the answers to the queries of the first round 
	we will be able to tell the one and only non-zero vector $u$ of $\uu$. This gives an algorithm of $n$ queries
	that matches the trivial lower bound mentioned earlier.)
	Let $T$ be any tree with vertex set $NZ$. Then in
  a second round of queries our algorithm asks the subspaces $L_{i,j,k}$ with
  $(i,j) \in E(T)$ and $1 \le k \le q-2$. We claim that after obtaining the
  answers to these queries, we are able to identify a vector $0 \neq
  u=(u_1,u_2,\ldots ,u_n) \in \uu$. Clearly, we have $u_i=0$ if and only $i \in
  Z$. As for any $i \in NZ$, we have $u_i\neq 0$, we obtain that for any pair
  $i,j \in NZ$ we have $(0,\ldots ,0,u_i,0,\ldots ,0,u_j,\ldots ,0) \in \lul_{i,j,k}$ for
  some $1 \le k \le q-1$. Thus by our queries of the second round, we will be
  able to tell to which such 1-subspace of $E_{i,j}$ the vector
  $(0,,,0,u_i,0,\ldots ,0,u_j,\ldots ,0)$ 
belongs.

Let us pick an arbitrary coordinate $x \in NZ$. We may assume that $u_x=1$ as
if not, then we can consider $u_x^{-1}u$ instead of $u$. Now for any $j \in
NZ$ with $(x,j) \in E(T)$ we can find out $u_j$ as there is exactly one vector
in $\lul_{x,j,k}$ with $x$-coordinate 1. As $T$ is connected and contains all
coordinates from $NZ$, we can determine all $u_i$'s 
with $i \in NZ$ one by one.
\hfill\mbox{}\qed
\end{proof}

One may obtain a bound in the non-adaptive case using a very similar strategy to that in the 2-round proof of \tref{adgen}. As this
time we have to ask all queries at a time, we have to make sure that no matter
what $NZ$ turns out to be we ask queries according to the edges of a connected
graph on $NZ$. To this end we do not have any other choice than to query for
all pairs $1 \le i<j\le n$. That is, we 
ask the separating system of the following subspaces:
\[
\{A_i: i=1,2,\ldots ,n\} \cup \{L_{i,j,k}: 1\le i<j \le n, 1\le k \le q-2\}.
\]

A proof identical to that in the adaptive case shows that this set of subspaces form a separating system.
This shows the bound $M(n,q) \le n+\binom{n}{2}(q-2)$. Thus we obtain that if $n$ is fixed, then $M(n,q)$ grows linearly in $q$. Our aim is not only to prove a similar statement for $n$, but to show that $M(n,q)$ grows linearly in $nq$. It is easy to see that the number of pairs of 1-subspaces separated by a subspace $U$ is maximized when $\dim(U)=n-1$. Thus a natural idea is to consider a set of randomly picked $(n-1)$-subspaces as candidate for a separating system of small size. This would yield the upper bound $M(n,q)=O(nq\log q)$. Another idea is to generalize what we used in the case of projective planes. If $n=3$, then the set of lines incident to at least one of 3 non-collinear points forms a separating system of size $3q-3$. Results of \sref{proj} show that this is not optimal, but is still of the right order of magnitude. Combining these two ideas, we obtain a proof of \tref{nonadgen}.\setcounter{section}{1}\setcounter{theorem}{3}

\begin{theorem}[upper bound]
$M(n,q) \le 2qn$.
\end{theorem}\setcounter{section}{3}\setcounter{theorem}{8}

\begin{proof} Let $V$ be an $n$-dimensional vector space over $\GF(q)$ and let
  $X_1,X_2, \ldots ,X_l$ be independent identically distributed random variables taking their values uniformly
  among all $(n-2)$-dimensional subspaces of $V$. For every $X_i$, there are
  exactly $q+1$ $(n-1)$-subspaces containing $X_i$, let us denote them by
  $X_{i,1},X_{i,2},\ldots ,X_{i,q+1}$. For any pair $\uu,\vv$ of 1-subspaces 
of $V$, let $S_{\uu,\vv}$ denote the indicator random variable of the event 
that $\uu$ and $\vv$ are not separated by
$X_{1,1},\ldots ,X_{1,q},X_{2,1},\ldots ,X_{2,q},\ldots ,X_{l,1},\ldots ,X_{l,q}$. 

\begin{claim}
\label{clm:count}
Let $\uu$ and $\vv$ be different 1-subspaces of $V$. Then the number of
$(n-2)$-subspaces $U$ of $V$ such that the family $\{U_1,U_2,\ldots ,U_{q+1}\}$ of all 
$(n-1)$-subspaces of $V$ containing $U$ does not separate $\uu$ and $\vv$ is
$(q-1){n-1 \brack n-3}-(q-2){n-2 \brack n-4}$.
\end{claim}

\textit{Proof of Claim.}
 Clearly, if $\uu,\vv \in U$, then $U_1,U_2,\ldots ,U_{q+1}$ do not
separate $\uu$ and $\vv$, while if exactly one of them lies in $U$, then 
 $U_1,U_2,\ldots ,U_{q+1}$ do separate them. If $\uu,\vv \notin U$, then $\uu$ and
 $\vv$ are not separated by $U_1,U_2,\ldots ,U_{q+1}$ if and only if
 $\dim(U,\langle \uu,\vv\rangle)=n-1$, that is if $U$ meets $\langle \uu,\vv
 \rangle$ in a 1-subspace different from both $\uu$ and $\vv$. As there are
 $q-1$ such 1-subspaces, the number of such $U$'s is $(q-1)\left({n-1 \brack
   n-3}-{n-2 \brack n-4}\right)$.
Thus the number of $(n-2)$-subspaces satisfying the condition of the claim is
${n-2 \brack n-4}+ (q-1)\left({n-1 \brack n-3}-{n-2 \brack n-4}\right)$ as claimed.
\hfill\mbox{}$\blacksquare$
\smallskip

\noindent By the above claim we obtain the expected value of $S_{\uu,\vv}$ satisfies
\[
\mathbb{E}(S_{\uu,\vv})=\left(\frac{(q-1){n-1 \brack 2}-(q-2){n-2 \brack 2}}{{n \brack 2}}\right)^l\le \left(\frac{(q-1)(q^{n-2}-1)}{q^n-1}\right)^l\le \frac{1}{q^l}.
\]
And thus if we set $l=2n$, then we have
\[
\mathbb{E}(\sum_{\uu,\vv}S_{\uu,\vv})
\le \binom{{n \brack 1}}{2}\frac{1}{q^l}\le 1/2.
\]
Therefore, there exists a collection of $2n$ $(n-2)$-dimensional subspaces such that the set of $(n-1)$-dimensional subspaces containing any of them is a separating family. Clearly, to separate pairs of 1-subspaces, it is enough to query $q$ of the $q+1$ $(n-1)$-subspaces containing a fixed $(n-2)$-subspace, and thus we have $M(n,q) \le 2nq$.
\hfill\mbox{}\qed
\end{proof}

To obtain the lower bound in \tref{nonadgen} we will use the following theorem of Katona \cite{K} about separating systems of sub\textit{sets} of an underlying set.

\begin{theorem} [Katona \cite{K}]
\label{thm:katona} Let $X$ be an $M$-element set and $\cA \subseteq 2^X$ be a separating system of subsets of $X$ such that for all $A \in\cA$ we have $|A| \le m$ for some integer $m<M/2$. Then the following inequality holds
\[
|\cA| \ge \frac{\log M}{\log e\frac{M}{m}}\frac{M}{m}.
\]
\end{theorem}

\setcounter{section}{1}\setcounter{theorem}{3}
\begin{theorem}[lower bound]
There exists an absolute constant $C>0$ such that for any positive integer $n$ and prime power $q$ the inequality $\frac{1}{C}qn \le M(n,q)$ holds. Moreover, if $q$ tends to infinity, then $(1-o(1))qn\le M(n,q)$ holds.
\end{theorem}\setcounter{section}{3}\setcounter{theorem}{8}
\begin{proof}
\tref{katona} can be applied to obtain the desired bound. Indeed, as mentioned in the Introduction, if $X$ is the set of all 1-subspaces of $V$ and the set $Q$ of all allowed queries is
\[
\left\{F \subset {V \brack 1}: \exists U \le V \textnormal{ with}\ F=\left\{\uu \in {V \brack 1}: \uu \le U\right\}\right\},
\]
then we can write $M={n \brack 1}=\frac{q^n-1}{q-1}$ and $m={n-1 \brack 1}=\frac{q^{n-1}-1}{q-1}$ since the largest ``meaningful'' query sets are those corresponding to $(n-1)$-subspaces of $V$. Substituting these values to the formula of \tref{katona} we obtain
\[
M(n,q) \ge \frac{\log M}{\log e\frac{M}{m}}\frac{M}{m}=\frac{\log\frac{q^n-1}{q-1}}{\log e\frac{q^n-1}{q^{n-1}-1}}\frac{q^n-1}{q^{n-1}-1}\ge (n-1)q\frac{\log q}{2+\log \frac{q^n-1}{q^{n-1}-1}}.
\] \hfill \mbox{ } \qed
\end{proof}


\section{Remarks}
We may formulate the dual searching problem: a hyperplane $H_0$ of $\PG(n-1,q)$ is marked, and we
can ask whether a subplane $H$ is contained in $H_0$; how many queries do we
need to identify $H_0$? Let us consider now the non-adaptive case in $\PG(n,q)$.
Suppose that we only ask points as queries. Thus we are to find a point set $S$ 
such that its intersection with any hyperplane is unique. Clearly, if the intersection of 
$S$ and any hyperplane contains $n$ points in general position, we are done.
Note that, however, this condition implies that any hyperplane is generated by its intersection
with $S$, which is clearly stronger than our original goal. Such a point set may be called a 
\emph{hyperplane generating set}. Let us denote the size of the smallest hyperplane generating
set by $\sigma(\PG(n,q))$, and denote the size of the smallest $n$-fold blocking set with respect to hyperplanes 
by $\tau_n^{n-1}(\PG(n,q))$. In case of $n=3$, that is projective planes, a hyperplane 
(line) generating set is just a double blocking set, thus $\sigma(\PG(2,q))=\tau_2^1(\PG(2,q))=\tau_2(\PG(2,q))$; 
furthermore, as seen in the remark after \tref{nolimit}, $M(3,q)$ is usually a bit smaller than $\tau_2(\PG(2,q))$.

In higher dimensions an $n$-fold blocking set with respect to hyperplanes is not necessarily a 
hyperplane generating set. Trivially, $\tau_n^{n-1}(\PG(n,q))\leq\sigma(\PG(n,q))$ and 
$M(n+1,q)\leq\sigma(\PG(n,q))$, but it is not clear how far these parameters are from each other if
$n\geq3$.

As any line intersects every hyperplane in at least one point, the union of $n$ pairwise nonintersecting 
lines is an $n$-fold blocking set with respect to hyperplanes.
We may also try to find a hyperplane generating set as the union of some lines. Let us say that
a set of lines is in \emph{higgledy-piggledy position} if their union is a hyperplane generating set.


Thus if we could find a set of $h(n,q)$ lines in higgledy-piggledy position in $\PG(n,q)$, then $M(n+1,q)\leq h(n,q)q$ would follow. 
For $n=2$, any three non-concurrent lines suffice; for $n=3$, one may take three lines of the same regulus of a hyperbolic
quadric and a fourth line disjoint from the quadric; for $n=4$, five lines turn out not to be enough \cite{FSzSz}. 
For $n\geq4$, we could not construct a small set of lines in higgledy-piggledy position so far.
The arising finite geometrical questions seem quite interesting \cite{FSz}.

\section{Acknowledgments}
We thank the anonymous referees for their helpful suggestions. Tam\'as H\'eger and Marcella Tak\'ats were supported by Hungarian National Scientific Fund (OTKA) Grant No.\ K~81310. Tam\'as H\'eger was also supported partially by ERC Grant No.\ 227701 DISCRETECONT. 
Bal\'azs Patk\'os was supported by OTKA Grant PD-83586 and the J\'anos Bolyai Research Scholarship of the Hungarian Academy of Sciences.


\begin{thebibliography}{99}
\bibitem{Bailey}
\textsc{R. F. Bailey}, Resolving sets for incidence graphs, 
Session talk at the 23rd British Combinatorial Conference, Exeter, 5th July 2011. 
Slides available online at \url{http://www.math.uregina.ca/~bailey/talks/bcc23.pdf} (last accessed June 21, 2013)
\bibitem{BC} 
\textsc{R. F. Bailey, P. J. Cameron}, Base size, metric dimension and other invariants of groups and graphs,
Bull. London Math. Soc. \textbf{43} 209--242 (2011)
\bibitem{BHSz}
\textsc{G. Bacs\'o, T. H\'eger, T. Sz\H onyi}, The $2$-blocking number and the upper chromatic
number of $\PG(2, q)$, to appear in J. Comb. Des.
\bibitem{Ball}
\textsc{S. Ball}, Multiple blocking sets and arcs in finite planes, 
J. London Math. Soc. (2)  \textbf{54} no. 3, 581--593 (1996)
\bibitem{BB}
\textsc{S. Ball, A. Blokhuis}, On the size of a double blocking set in $\PG(2,q)$, Finite Fields
Appl., \textbf{2} 125--137 (1996)
\bibitem{BlBr}
\textsc{A. Blokhuis, A. E. Brouwer}, Blocking sets in Desarguesian projective planes,
Bull. London Math. Soc. \textbf{18} no. 2, 132--134 (1986)
\bibitem{BSSz} 
\textsc{A. Blokhuis, L. Storme, T. Sz\H onyi}, Lacunary polynomials, multiple blocking sets and Baer subplanes, 
J. London Math. Soc. (2)  {\bf 60} no. 2, 321--332 (1999)
\bibitem{DH}
\textsc{D.-Z. Du, F.K. Hwang}, \textbf{Combinatorial Group Testing and its Applications}, 2nd ed. (English)
Series on Applied Mathematics (Singapore). 12. Singapore: World Scientific. xii, 323 p. (2000). 
\bibitem{FSz}
\textsc{Sz. L. Fancsali, P. Sziklai}, Lines in higgledy-piggledy position. \textit{Submitted.}
\bibitem{FSzSz}
\textsc{Sz. L. Fancsali, P. Sziklai, T. Sz\H{o}nyi}, Personal communication (2013).
\bibitem{Harrach} 
\textsc{N.~V.~Harrach}, Unique reducibility of multiple blocking sets,
J. Geometry \textbf{103} 445--456 (2012)
\bibitem{HT}
\textsc{T. H\'eger, M. Tak\'ats}, Resolving Sets and Semi-Resolving Sets in Finite Projective Planes, Electronic J. of Combinatorics, \textbf{19} no.~4, P30 (2012)
\bibitem{Hirsch} 
\textsc{J.~W.~P.~Hirschfeld}, \textbf{Projective geometries over finite fields},
Clarendon Press, Oxford, 1979, 2nd edition, 1998.
\bibitem{K}
\textsc{G. Katona}, On separating systems of a finite set, J. Combin. Theory \textbf{1} 174--194 (1966)
\bibitem{Sz}
\textsc{T. Sz\H onyi}, Blocking sets in Desarguesian affine and projective planes. Finite Fields
and Appl. \textbf{3} 187--202 (1997)
\end{thebibliography}
\end{document}